\documentclass[12pt, twoside]{amsart} 
\date{\today, version 1.20}

\usepackage{pxfonts}

\usepackage{latexsym}
\usepackage{amssymb}
\usepackage{amsthm}
\usepackage{amscd}
\usepackage{enumerate} 
\usepackage{amssymb} 
\usepackage{mathrsfs}          
\usepackage[all]{xy}
\usepackage{color}

  \newtheorem{theorem}{Theorem}[section]
  \newtheorem{lemma}[theorem]{Lemma}
  \newtheorem{proposition}[theorem]{Proposition}
  \newtheorem{corollary}[theorem]{Corollary}
  
  \newtheorem{conjecture}[theorem]{Conjecture}

  \newtheorem{definition}[theorem]{Definition}
  \newtheorem{example}[theorem]{Example}

\newtheorem{remark}[theorem]{Remark}

\theoremstyle{remark}

\numberwithin{equation}{section}

\usepackage[all]{xy}

\begin{document}

\title[The Mukai-type conjecture]{Effective non-vanishing and the Mukai-type conjecture}

\author[Y. Gongyo]{Yoshinori Gongyo}
\address{Graduate School of Mathematical Sciences, The University of Tokyo, 3-8-1 Komaba,
Meguro-ku, Tokyo, 153-8914, Japan.}
\email{gongyo@ms.u-tokyo.ac.jp}


\maketitle

\begin{abstract}We introduce some invariants of Fano varieties and propose a Mukai-type conjecture which characterizes the product of projective spaces.  Moreover, we prove that the Ambro--Kawamata effective non-vanishing conjecture implies the Mukai-type conjecture. 

\end{abstract}

\setcounter{tocdepth}{1} 
\tableofcontents

\section{Introduction}Throughout this article, we will work over the field of complex numbers $\mathbb C$. 
We will make use of the standard notation as in \cite{kamama}, \cite{ko},
\cite{bchm}, and \cite{komo}. 
In this short article, we study the characterization of the product of the projective spaces for Fano varieties. The following is a very famous conjecture by Shigeru Mukai:

\begin{conjecture}[{\cite{mukai}}]\label{mukai conjecture}Let $X$ be a $d$-dimensional smooth Fano varieties,  $\rho(X)$ the Picard number, and 
$$i_X:= \max\{r \in \mathbb{ Z}| -K_X \sim_{\mathbb{Z}} r H \text{ for some Cartier divisor $H$} \}$$ be the Fano index. Then it holds that 

$$d+\rho(X)-i(X) \cdot \rho(X) \geq 0.
$$
Moreover the above is equal if and only if $X \simeq \mathbb{P}^{i(X)-1} \times \cdots  \times \mathbb{P}^{i(X)-1}. $
\end{conjecture}

The classical approach to the Mukai conjecture is to study the geometry of rational curves cf. \cite{w-mukai}, \cite{bcdd}...  In fact, in the series of these studies, they prove some generalization of the above Mukai conjecture, where Fano index is replaced by the {\it pseudo index} which is the minimal length of rational curves on the variety. 

 In this article, we introduce some invariants of Fano varieties and propose the new Mukai-type conjecture from the motivation of the Mukai conjecture. We call the invariant the {\it total index}, which is  the maximal length of nef-partitions on Toric geometry that appeared in \cite{ba}, \cite{bn}
(cf. Section \ref{total indices:section}). The total index coincides with the Fano index when the Picard number is one. We propose the new conjecture by replacing $\rho(X) \cdot i(X)$ with the total index $\tau_X$ in Conjecture \ref{mukai conjecture} as follows:

 \begin{conjecture}\label{gamma Mukai conjecture}Let $X$ be a $d$-dimensional smooth Fano varieties,  $\rho(X)$ the Picard number, and $\tau_X$ be the total index in Definition \ref{total-index}. Then it holds that 
 
 $$\mathrm{dim}\,X + \rho(X) - \tau_X \geq 0 $$ 
 and the equality holds if and only if the product of projective spaces, i.e. $X \simeq \mathbb{P}^{d_1} \times \cdots  \times \mathbb{P}^{d_{\rho(X)}}$ for some $d_1, d_2, \dots, d_{\rho(X)} \in \mathbb{Z}_{>0}.$
 
 \end{conjecture}
 
 Note that Conjectures \ref{mukai conjecture} and  \ref{gamma Mukai conjecture} coincide with the Kobayashi--Ochiai theorem \cite{kooc} when the Picard number is one. 
 
 The main ingredient of this article is to prove Conjecture \ref{gamma Mukai conjecture} under assuming that the following Ambro--Kawamata effective non-vanishing conjecture is true:
 
 \begin{conjecture}\label{effective non-vanishing}Let $(X,B)$ be a projective klt pair and $L$ a nef line bundle on $X$ such that $L-(K_X+B) $ is ample. 

Then $H^0(X, L) \not =0$. 

\end{conjecture}

Namely the following is the main theorem in this article:

\begin{theorem}
\label{K implies M}Assume that Conjecture \ref{effective non-vanishing} holds. Then Conjecture \ref{gamma Mukai conjecture} holds. 

\end{theorem}

We note that we study Conjecture \ref{gamma Mukai conjecture} from the viewpoint of Shokurov's complexity for generalized pairs in \cite{gm}. 

\section*{Aknowledgement}
The author was partially supported by grants JSPS KAKENHI $\#$16H02141, 17H02831, 18H01108, 19KK0345, 20H00111, and JSPS Bilateral Joint Research Projects JPJSBP120219935. He thanks Professors Taku Suzuki,  Atsushi Ito, Joaqu\'in Moraga, Yuji Odaka, Kenji Matsuki, and Shigeru Mukai for the discussion. I appreciate A. Ito for informing references \cite{ba}, \cite{bn}. Moreover, he and J. Moraga study this project together from the study of generalized complexities in \cite{gm}. He also thanks Professor Osamu Fujino for informing the right references for the fact of quasi-log schemes and pointing out mistakes in the previous version.  
\section{Preliminaries}

 In this section, we recall some notations of generalized pairs and quasi-log schemes. 
For more detail, we refer the reader to \cite{Bir20} and \cite{ambro-quasi},  \cite{fujino-book}. First, we recall the definition of generalized log canonical and kawamata log terminal.

\begin{definition}\label{defi:gp}
A \textit{generalized pair} $(X,B+M)$ consists of
\begin{itemize}
\item a normal projective variety $X$, 
\item an effective $\mathbb{R}$-divisor $B$ on $X$, and 
\item a b-$\mathbb{R}$-Cartier b-divisor over $X$ represented by some projective birational morphism $\varphi: X' \to X$ and 
a nef $\mathbb{R}$-Cartier divisor $M'$ on $X'$
\end{itemize}
such that $M = \varphi_* M'$ and $K_{X}+B+M$ is $\mathbb{R}$-Cartier. 
\end{definition}

\begin{definition}
\label{def:gen-lc}
{\em 
Let $(X, B+M)$ be a generalized pair.
Let $Y\rightarrow X$ be a projective birational morphism.
We can write
\[
K_Y+B_Y+M_Y=\pi^*(K_X+B+M),
\]
for some divisor $B_Y$ on $Y$.
The {\em generalized log discrepancy}
of $(X,B+M)$ at a prime divisor $E\subset Y$, denoted by $a_E(X,B+M)$
is defined to be $1-{\rm coeff}_E(B_Y)$.

A generalized pair $(X, B+M)$ is said to be {\em generalized log canonical} (or glc for short) if all its generalized log discrepancies are non-negative.
A generalized pair $(X,B+M)$ is said to be {\em generalized Kawamata log terminal} (or gklt for short) if all its generalized log discrepancies are positive.
In the previous definitions, if the analogous statement holds for $M=0$, then we drop the word ``generalized".

}
\end{definition}

Moreover, we recall the notion of the quasi-log structure by Ambro--Fujino:

\begin{definition}[Quasi-log canonical pairs]\label{quasi-log}

Let $X$ be a scheme and let $\omega$ be an 
$\mathbb R$-Cartier divisor (or an $\mathbb R$-line bundle) on $X$. 
Let $f:Z\to X$ be a proper morphism from a globally embedded 
simple normal 
crossing pair $(Z, \Delta_Z)$, i.e. $Z$ is a simple normal crossing divisor on some smooth variety $M$, and  $\Delta_Z=\Delta_M|_Z$, where  $\Delta_M$ is some divisor  on $M$ such that $\mathrm{Supp}\,(\Delta_M+Z)$ is simple normal crossing and  $\mathrm{Supp}\,\Delta_M$ does not contain any components of  $Z$. Note that $\Delta_Z$  is an $\mathbb{R}$-divisor which is not necessarily effective.     If the natural map 
$\mathcal O_X\to f_*\mathcal O_Z(\lceil -(\Delta_Z^{< 1})\rceil)$ is an 
isomorphism, $\Delta_Z=\Delta^{\leq 1}_Z$, 
and $f^*\omega\sim _{\mathbb R} K_Z+\Delta_Z$, 
then $\left(X, \omega, f:(Z, \Delta_Z)\to X\right)$ 
is called a {\em{quasi-log canonical pair}} 
({\em{qlc pair}}, for short). 
Let $C$ be a closed subvariety of $X$. We say that $C$ is a {\it qlc center} if $C$ is not an irreducible component of $X$ and an image of some stratum of $(Z, \lfloor \Delta_Z \rfloor )$. 
We denote the union of all qlc centers by  $\mathrm{Nqklt(X, \omega)}$. 

We call the triple  {\em{quasi-kawamata log terminal pair}} 
({\em{qklt pair}}, for short) when  $\mathrm{Nqklt(X, \omega)}=\emptyset$ 
. If there is no danger of confusion, 
we simply say that $[X, \omega]$ is a qlc pair (resp. qklt pair). 
\end{definition}

\begin{lemma}\label{rem-qklt}Let $[X, \omega]$ be a projective  qklt, then $X$ is normal and there exists an $\mathbb{R}$-divisor $B+M$ such that $ \omega \sim_{\mathbb{R}}K_X+B+M $ and $(X, B+M)$ is generalized klt. Moreover, when $\omega$ is a $\mathbb{Q}$-divisor, so are $B$ and $M$. 
\end{lemma}

\begin{proof}From the definition of qklt, it follows that  $\mathrm{Nqklt(X, \omega)}= \emptyset$. 
By \cite[Lemma 6.9]{fujino-book}, we see that $X$ is normal. The latter assertion follows from \cite[Theorem 1.3]{fujino-on-quasi}. 

\end{proof}

\section{Total indices}\label{total indices:section}In this section, we define the total index of Fano varieties and compute some examples. In \cite{ba}, \cite{bn}, the invariants for Toric varieties appear in the context of the Mirror symmetry.

 \begin{definition}\label{total-index}Let $X$ be a smooth Fano variety. We define {\it the total index} $\tau_X$ of  $X$ as the maximal of $\sum a_i$ for  a decomposition 
 
 $$-K_X= \sum a_i L_i,$$
 where $L_i$ are nef line bundles which is not numerically trivial and $a_i \in \mathbb{Z}_{>0}$.  Note that we allow $L_i=L_j$ for $i \not =j$.  
 
 \end{definition} 
 
 It seems to be no direct relationship with the product of the Fano index and the Picard number although  we propose Conjecture \ref{gamma Mukai conjecture} motivated with Conjecture \ref{mukai conjecture} ;
 
 \begin{example}\label{eg1}Let $X$ be a three point blow-up of $\mathbb{P}^2$. Then $\tau_X=3$ but $\rho(X)=4$ and the Fano index $i(X)=1$. Thus $\tau_X < \rho(X) i(X)$. 
 
 \end{example}

  \begin{example}\label{eg2}Let $X$ be a one point blow-up of $\mathbb{P}^2$. Then $\tau_X=3$ but $\rho(X)=2$ and the Fano index $i(X)=1$. Thus $\tau_X > \rho(X) i(X)$. 

 \end{example}

\section{Ambro--Kawamata's effective non-vanishing conjecture}In this section, we discuss extending Conjecture \ref{effective non-vanishing} beyond kawamata log terminal singularities.

We expect that  Conjecture \ref{effective non-vanishing} is valid for generalized klt pairs or quasi-log canonical pairs:

\begin{lemma}\label{effective non-vanishing:g}Assume that Conjecture \ref{effective non-vanishing} holds. Let $[X,\omega]$ be a projective quasi-log canonical pair (resp. $(X,B+M)$ is a projective generalized klt pair) and $L$ a nef line bundle on $X$ such that $L-\omega$ (resp. $L-(K_X+B+M)$) is ample. 
Then $H^0(X, L) \not =0$. 

\end{lemma}

\begin{proof}We shall prove by induction on dimension. By the adjunction and vanishing theorem \cite[Theorem 6.3.4(2)]{fujino-book}, we may assume that $[X,\omega]$ has no quasi-lc center, then $[X, \omega]$ has generalized klt structure by Lemma \ref{rem-qklt}, i.e. $ \omega\sim_{\mathbb{R}}K_X+B+M $ for some generalized klt pair $(X, B+M)$. Moreover, take a general effective $\mathbb{R}$-ample divisor $H$ such that there exists an effective $\mathbb{R}$-divisor $B'$ with $B+M+\varepsilon H \sim_{\mathbb{R}} B'$ for some small $\varepsilon>0$ and that $(X, B')$ is klt.    Thus, we may assume that $M=0$. Thus the conclusion holds. 

\end{proof}

Now we discuss some version of the ladder arguments for Fano varieties cf. \cite{ambro-fano}, \cite{kawamata-effective}.

\begin{lemma}\label{adding divisor} Let $[X, \omega]$ be a normal quasi-log canonical pair and $D$ be an effective $\mathbb{R}$-Cartier divisor on $X$. Let $l$ be the quasi-log canonical thresholds of $D$ for $[X, \omega]$, i.e. 
$$l=\sup\{s\geq 0| [X,  \omega+sD ] \text{ has quasi-log canonical structure}\}. $$ Then for an associated generalized log canonical structure $(X, B+M)$ of $[X,\omega]$, the divisors  
$$B'=B+lD \text{ and } M'=M
$$
give a generalized log canonical structure $(X, B'+M')$ on $[X,\omega+lD].$

\end{lemma}

\begin{proof}See the proof of \cite[Theorem 1.7]{fujino-slc fibraion}. 

\end{proof}

\begin{proposition}\label{choice of lc boundary}Assume that Conjecture \ref{effective non-vanishing} holds. Let $[X,\omega]$ be a projective normal quasi-log canonical scheme such that $-\omega$ is nef with a generalized lc structure $(X, B+M)$ such that $\omega\sim_{\mathbb{R}} K_X+B+M$ and $X$ is of Fano type, i.e. there exists a big $\mathbb{Q}$-divisor $\Delta \sim_{\mathbb{Q}} -K_X$ such that $(X, \Delta)$ is klt.  Suppose that we have a decomposition 
 
 $$-(K_X+B+M) \equiv \sum_{i=1}^r a_i L_i,$$
 where $L_i$ are nef line bundles such that $L_i \not \equiv 0$ for any $i$.  Then there exists  a reduced divisor $S_{i,j} \sim_{\mathbb{Z}} L_i $ for $j=1, \cdots , \lceil a_i  \rceil$ and for any $i$ such that 
 $$(X, B+M+\sum_{i=1}^r \sum_{j=1}^{\lfloor a_i \rfloor}S_{i,j}+ \sum_{i}\{a_i\} S_{i, \lceil a_i  \rceil})$$  is log canonical. 
\end{proposition}

\begin{proof} First we shall show the  proposition when $-(K_X+B+M)$ is ample. 
By Conjecture \ref{effective non-vanishing} and Lemma \ref{effective non-vanishing:g},we have some effective divisor $S_1  \sim_{\mathbb{Z}} L_1$. Let $l$ be the log canonical thresholds of $S_1$ for the pair $(X,B+M)$. We consider the case of  $a_1 \geq 1$.
We assume that  $l<1$.  Then 
$$L_1-(K_X+B+M+lS_1)\equiv   (1+a_1-l) S_1 + \sum_{i>1} a_i L_i$$
is still ample since $  1+a_1-l  >0.$ Let  $C$ be a minimal lc center of $(X,B+M+lS_1)$.
By  applying the  vanishing theorem  \cite[Theorem 6.3.4(2)]{fujino-book} cf.\cite{clx}, we obtain the surjectivity of 

$$
H^0(X, L_1) \to H^0(C, L_1).
$$
Note that this minimal lc center is a minimal qlc center from Lemma \ref{adding divisor} and \cite[Theorem 1.7]{fujino-slc fibraion}.
By the adjunction formula of \cite[Theorem 1.1]{fujino--hashizume}, we have $-(K_X+B+M+lS_1)|_C = -(K_C+B_C+M_C)$ for some $B_C+M_C$  such that  $(C,B_C+M_C )$ is generalized klt. Since $  a_1-l  >0,$ $-(K_C+B_C+M_C)$ is ample. 

By Conjecture \ref{effective non-vanishing} (cf. Lemma \ref{effective non-vanishing:g}) for $C$, $$H^0(C, L_1) \not =0.$$
Thus if we choose $S_1$ as general, $C$ must be contained in the base locus of $|L_1|$. Therefore this is the contradiction and it holds that $l=1$. Thus by induction on $\lfloor a_1 \rfloor$, we may assume that $a_1 <1$. Then, for the log canonical thresholds  $l$ of $S_1$ for the pair $(X,B+M)$, if $l<a_1$, we make the contradiction from the same arguments as the above. So $l\geq a_1$. Thus by induction on $r$, we see the desired proposition.  

Next, we discuss the general case. By the base point free theorem for quasi-log pairs (cf. \cite[Theorem 6.5.1]{fujino-book}),  we have the semi-ample fibration $\varphi: X \to Y$ of $-\omega.$ Indeed let $f:(Z,B_Z) \to X$ be a quasi-log structure of $(X, \omega)$. Then we have an ample $\mathbb{R}$-Cartier divisor $-\omega_Y$ such that $-\omega=\varphi^*(-\omega_Y).$   Thus the induced morphism  $ f \circ \varphi: (Z,B_Z) \to Y$ gives a quasi-log structure on $(Y, \omega_
{Y'}).$ Note that  we  give a generalized log canonical structure $(Y, B_Y+M_Y)$ on $\omega_Y$ from \cite[Theorem 1.7]{fujino-slc fibraion}.   Moreover, since $L_i \sim_{f} 0$ by the base point free theorem, we have some nef Cartier divisor $L_{i,Y}$ on $Y$ such that $L_i=\varphi^*L_{i,Y}$ for any $i$. 
Since $-\omega_Y$ is ample, we have 
a reduced divisor $S_{i,j,Y} \sim_{\mathbb{Z}} L_{i,Y} $ for $j=1, \cdots , \lceil a_i  \rceil$ and for any $i$ such that 
 $$(Y, B_Y+M_Y+\sum_{i} \sum_{j=1}^{\lfloor a_i \rfloor}S_{i,j,Y}+ \sum_{i}\{a_i\} S_{i, \lceil a_i  \rceil, Y})$$  is log canonical. Now since 
 $$K_X+B+M+\sum_{i} \sum_{j=1}^{\lfloor a_i \rfloor}S_{i,j}+ \sum_{i}\{a_i\} S_{i, \lceil a_i  \rceil}=\varphi^*(K_Y+B_Y+M_Y+\sum_{i} \sum_{j=1}^{\lfloor a_i \rfloor}S_{i,j,Y}+ \sum_{i}\{a_i\} S_{i, \lceil a_i  \rceil, Y})$$
by letting $S_{i,j}=\varphi^*S_{i,j,Y}$, we apply the inversion of adjunction for the fibration cf. \cite[Proposition 4.16]{filipazzi} and see that  $$(X, B+M+\sum_{i} \sum_{j=1}^{\lfloor a_i \rfloor}S_{i,j}+ \sum_{i}\{a_i\} S_{i, \lceil a_i  \rceil})$$  is log canonical.  Thus we prove the proposition. 
\end{proof}

\begin{corollary}\label{eff non-van dime 2}Let $[X,\omega]$ be a projective normal quasi-log canonical scheme such that $-\omega$ is nef with a generalized lc structure $(X, B+M)$ such  that $\omega\sim_{\mathbb{R}} K_X+B+M$ and $X$ is of Fano type.   Suppose that we have a decomposition 
 
 $$-(K_X+B+M) \equiv \sum a_i L_i,$$
 where $L_i$ are nef line bundles. If $X$ is a smooth  Fano $3$-fold  or $\mathrm{dim}\,X=2$, then there exists  a reduced divisor $S_{i,j} \sim_{\mathbb{Z}} L_i $ for $j=1, \cdots , \lceil a_i  \rceil$ and for any $i$ such that 
 $$(X, B+M+\sum_{i} \sum_{j=1}^{\lfloor a_i \rfloor}S_{i,j}+ \sum_{i}\{a_i\} S_{i, \lceil a_i  \rceil})$$  is log canonical. 

\end{corollary}

\begin{proof}
When $\mathrm{dim}\, X=2$, the fully general case of Conjecture \ref{effective non-vanishing} is known in \cite[Theorem 5.1]{kawamata-effective}. And for smooth Fano $3$-folds, it is known by \cite[Theorem 1.3]{BH}.
\end{proof}

\begin{remark}\label{glc}  Lemma \ref{effective non-vanishing:g}, Proposition \ref{choice of lc boundary}, and Corollary \ref{eff non-van dime 2} can be proved for generalized log canonical pairs by using the results in \cite{filipazzi} and the recent preprint \cite{clx}. 
 
\end{remark}

\section{Proof of Theorem \ref{K implies M}}

\begin{proof}We may assume that $\mathrm{dim}\,X + \rho_X - \tau_X \leq 0.$ By Proposition \ref{choice of lc boundary}, we obtain   a reduced divisor $S_{i,j} \sim_{\mathbb{Z}} L_i $ for $j=1, \cdots , a_i  $ and for any $i$ such that 
 $$(X, \sum_{i} \sum_{j=1}^{ a_i }S_{i,j} )$$  is log canonical. Let $\Delta=\sum_{i} \sum_{j=1}^{a_i }S_{i,j}$. Then the global complexity of $(X, \Gamma)$ is not less than $\mathrm{dim}\,X + \rho_X$. Thus by \cite{bmsz},
 $$(X, \sum_{i} \sum_{j=1}^{ a_i }S_{i,j})$$ is a Toric pair and 
 $$\mathrm{dim}\,X + \rho_X - \tau_X = 0.$$ Since $\sum a_i= \mathrm{dim}\,X + \rho_X$, any torus invariant divisors appear in $\{S_{i,j} \} $. Thus we conclude every effective divisor is nef on $X$. Such a smooth Toric variety is only a product of projective spaces (cf.\cite[Proposition 5.3]{fs}).
 \end{proof}

\end{document}